\documentclass[12pt]{article}%
\usepackage{amsmath}
\usepackage{amsfonts}
\usepackage{amssymb}
\usepackage{graphicx}%
\providecommand{\U}[1]{\protect\rule{.1in}{.1in}}
\newtheorem{theorem}{Theorem}

\newtheorem{lemma}[theorem]{Lemma}

\newenvironment{proof}[1][Proof]{\noindent\textbf{#1.} }{\ \rule{0.5em}{0.5em}}
\begin{document}

\title{On the asymptotic Plateau's problem for CMC hypersurfaces in hyperbolic space}
\author{Jaime Ripoll}
\date{}
\maketitle

\begin{abstract}
Let $\mathbb{R}_{+}^{n+1}$ \ be the half-space model of the hyperbolic space
$\mathbb{H}^{n+1}.$ It is proved that if $\Gamma\subset\left\{  x_{n+1}%
=0\right\}  \subset\partial_{\infty}\mathbb{H}^{n+1}$ is a bounded $C^{0}$
Euclidean graph over $\left\{  x_{1}=0,\text{ }x_{n+1}=0\right\}  $ then,
given $\left\vert H\right\vert <1,$ there is a complete, properly embedded,
CMC $H$ hypersurface $S$ of $\mathbb{H}^{n+1}$ such that $\partial_{\infty
}S=\Gamma\cup\left\{  x_{n+1}=+\infty\right\}  .$ Moreover, $S$ is a
horizontal graph over the totally geodesic hypersurface $\mathbb{H}%
^{n}=\left\{  x_{1}=0\right\}  .$

Horizontal graphs are limit of radial graphs. Thus, this result can be seen as
a limit case of the existence theorem proved by B. Guan and J. Spruck in
\cite{GS} on CMC $\left\vert H\right\vert <1$ radial graphs with prescribed
$C^{0}$ asymptotic boundary data.

\end{abstract}

\section{Introduction}

\qquad Let $\mathbb{R}_{+}^{n+1}$ be the half-space model of the hyperbolic
space $\mathbb{H}^{n+1}.$ B. Guan and J. Spruck in \cite{GS}, Theorem 1.8,
proved that if $\Gamma$ is a compact starshaped $C^{0}$ hypersurface of
$\mathbb{R}^{n}=\left\{  x_{n+1}=0\right\}  \subset\partial_{\infty}%
\mathbb{H}^{n+1}$ then, given $\left\vert H\right\vert <1,$ there exists a
complete, properly embedded, $C^{\infty},$ constant mean curvature (CMC) $H$
hypersurface $S$ in $\mathbb{H}^{n+1}$ such that $\overline{S}=S\cup\Gamma$ is
a $C^{0}$ hypersurface with boundary in $\overline{\mathbb{H}}^{n+1}%
=\mathbb{H}^{n+1}\cup\partial_{\infty}\mathbb{H}^{n+1}.$ The hypersurface $S$
is constructed in \cite{GS} as a radial graph of an entire solution of a
PDE\ defined on an upper hemisphere $\mathbb{S}_{+}^{n}$ of the unit sphere
$\mathbb{S}^{n}$ centered at the origin of $\mathbb{R}_{+}^{n+1}.$

\medskip

Our main result is an extension of Theorem 1.8 of \cite{GS} to horizontal
graphs, a limit case of radial graphs. To present it in a precise and clear
form it is convenient to state Guan-Spruck result in intrinsic terms.

\medskip

Using the conformal structure of $\partial_{\infty}\mathbb{H}^{n+1}$ (see
first paragraph of Section \ref{hy})$,$ it is easy to see that the starshaped
and compactness properties of $\Gamma$ are equivalent to the existence of two
distinct points $p_{1},p_{2}\in\partial_{\infty}\mathbb{H}^{n+1},$ not
belonging to $\Gamma,$ such that any arc of circle having $p_{1}$ and $p_{2}$
as ending points intersects $\Gamma$ at one and only one point. Moreover, the
radial property of $S$ is equivalent to $S$ being the \emph{hyperbolic}
Killing graph of a function defined in a totally geodesic hypersurface of
$\mathbb{H}^{n+1}$. Precisely: There exist a totally geodesic hypersurface
$\mathbb{H}^{n}$ of $\mathbb{H}^{n+1}$, a Killing field $X$ which orbits in
$\mathbb{H}^{n+1}$ are hypercycles orthogonal to $\mathbb{H}^{n}$ and in
$\partial_{\infty}\mathbb{H}^{n+1}$ are the arcs of circles from $p_{1}$ to
$p_{2},$ and a function $u\in C^{\infty}\left(  \mathbb{H}^{n}\right)  \cap
C^{0}\left(  \overline{\mathbb{H}}^{n}\right)  $ such that%
\[
\overline{S}=\left\{  \Psi_{X}(u(x),x)\text{
%TCIMACRO{\TEXTsymbol{\vert} }%
%BeginExpansion
$\vert$
%EndExpansion
}x\in\overline{\mathbb{H}}^{n}\right\}  ,
\]
where $\Psi_{X}\left(  t,x\right)  ,$ $t\in\mathbb{R}$ is the flow of $X$
through $x$, $\Psi_{X}\left(  0,x\right)  =x,$ $x\in\overline{\mathbb{H}%
}^{n+1}$. The fact that $S$ has CMC $H$ is equivalent to that of $u$ being a
solution of the PDE (\ref{pde}).

\medskip

In the present paper we reprove the above theorem of Guan-Spruck (Theorem
\ref{hyp} below) and extend it to include the limit case where we have
$p_{1}=p_{2}.$ In this case, $p:=p_{1}=p_{2},$ we assume that $\Gamma$ is
contained in a pinched annulus bounded by two hyperspheres, that this, there
are hyperspheres $E_{1}$ and $E_{2}$ of $\partial_{\infty}\mathbb{H}^{n+1}$
which are tangent to $p\in E_{1}\cap E_{2}$ and $\Gamma\subset U_{1}\cap
U_{2}$ where $U_{i}\subset\partial_{\infty}\mathbb{H}^{n+1}$ is the closure of
the connected component of $\partial_{\infty}\mathbb{H}^{n+1}\backslash E_{i}$
that contains $E_{j},$ $i\neq j$ (note that this condition is the limit of the
case in which $\Gamma$ is contained in an annulus of $\partial_{\infty
}\mathbb{H}^{n+1}$ bounded by two hyperspheres$,$ condition which is trivially
satisfied in Theorem 1.8 of \cite{GS} from the assumption that $p_{1}%
,p_{2}\notin\Gamma$ and that $\Gamma$ is compact)$.$ Moreover, we require that
any circle passing through $p$ orthogonal to $E_{1},$ intersects $\Gamma$ at
one and only one point (note also that this condition is the limit of the
starshaped condition)$.$

Under the above assumptions it is proved (Theorem \ref{par}) the existence of
a properly embedded CMC $\left\vert H\right\vert <1$ hypersurface $S$ of
$\mathbb{H}^{n+1}$ having $\Gamma$ as asymptotic boundary; moreover, $S$ is a
parabolic graph, that is, there exist a totally geodesic hypersurface
$\mathbb{H}^{n}$ of $\mathbb{H}^{n+1}$, a Killing field $Y$ vanishing at $p$
which orbits in $\mathbb{H}^{n+1}$ are horocycles orthogonal to $\mathbb{H}%
^{n}\ $and in $\partial_{\infty}\mathbb{H}^{n+1}$ are the circles through $p$,
orthogonal to $E_{1}$ at $p,$ and a function $u\in C^{\infty}\left(
\mathbb{H}^{n}\right)  \cap C^{0}\left(  \overline{\mathbb{H}}^{n}%
\backslash\left\{  p\right\}  \right)  $ such that
\[
\overline{S}=\left\{  \Psi_{Y}(u(x),x)\text{
%TCIMACRO{\TEXTsymbol{\vert} }%
%BeginExpansion
$\vert$
%EndExpansion
}x\in\overline{\mathbb{H}}^{n}\right\}  ,
\]
where $\Psi_{Y}$ is the flow of $Y$.

We note that in the half-space model $\mathbb{R}_{+}^{n+1}$ of $\mathbb{H}%
^{n+1},$ this last result can be stated as follows. Take $Y$ as $Y(x_{1}%
,...,x_{n+1})=(1,0,...,0)$, $\Psi_{Y}\left(  t,\left(  x_{1},...,x_{n+1}%
\right)  \right)  =\left(  x_{1}+t,...,x_{n+1}\right)  .$ Then if
$\Gamma\subset\left\{  x_{n+1}=0\right\}  $ is a bounded $C^{0}$ graph over
$\left\{  x_{1}=0,\text{ }x_{n+1}=0\right\}  $ then there is a hypersurface
$S$ with prescribed mean curvature $H$ which is a horizontal graph over the
totally geodesic hypersurface $\mathbb{H}^{n}=\left\{  x_{1}=0\right\}  $ such
that $\partial_{\infty}S=\Gamma\cup\left\{  x_{n+1}=+\infty\right\}  .$ Note
that in this model $p=\left\{  x_{n+1}=+\infty\right\}  .$

\medskip

The above results can be summarized in the following short statement: Let $Z$
be a Killing fieId of the hyperbolic or parabolic type in $\mathbb{H}^{n+1}$
orthogonal to a totally geodesic hypersurface $\mathbb{H}^{n}$ of
$\mathbb{H}^{n+1}.$ If $\Gamma\subset\partial_{\infty}\mathbb{H}^{n+1}$ is a
compact, embedded, topological hypersurface which is a $Z-$Killing graph over
$\partial_{\infty}\mathbb{H}^{n}$ then, given $\left\vert H\right\vert <1,$
there is a complete, properly embedded, CMC $H$ hypersurface $S$ of
$\mathbb{H}^{n+1}$ which is a $Z-$Killing graph over $\mathbb{H}^{n}$ such
that $\partial_{\infty}S=\Gamma.$

\medskip

Our approach for proving Theorems \ref{hyp} and \ref{par} is based on Perron's
method for the CMC equation of Killing graphs (Theorem \ref{perron}), which
validity is proved by using Theorems 1 and 2 of \cite{DLR} (Theorems 1 and 2
stated below). With this method we obtain a simple and unified proof of both
Theorems \ref{hyp} and \ref{par}.

\section{Perron's method for the mean curvature equation of Killing graphs}

\qquad Let $N^{n+1}$ denote a $(n+1)$-dimensional Riemannian manifold. Assume
that $N$ admits a Killing vector field $Z$ with no singularities, whose
orthogonal distribution is integrable and whose integral lines are complete.
Fix an integral hypersurface $M^{n}$ of the orthogonal distribution. It is
easy to see that $M$ is a totally geodesic submanifold of $N$. Denote by
$\Psi\colon\mathbb{R}\times N\rightarrow N$ the flow of $Z$. Notice that
$\gamma=1/\left\langle Z,Z\right\rangle \>$can be seen as a function in $M$
since $Z\gamma=0$ by the Killing equation. Moreover, the solid cylinder
$\Psi(\mathbb{R}\times M)$ with the induced metric has a warped product
Riemannian structure $M\times_{\rho}\mathbb{R}$ where $\rho=1/\sqrt{\gamma}$.

Given a function $u$ on an open subset $\Omega$ of $M$ the associated
\emph{Killing graph} is the hypersurface
\[
\mbox{Gr}(u)=\{\Psi(u(x),x)\text{
%TCIMACRO{\TEXTsymbol{\vert} }%
%BeginExpansion
$\vert$
%EndExpansion
}x\in\Omega\}.
\]
It is shown in \cite{DHL} that $\mbox{Gr}(u)$ has CMC $H$ if and only if $u\in
C^{2}(\Omega)$ satisfies
\begin{equation}
Q_{H}\left(  u\right)  =\text{div}\Bigg(\frac{\nabla u}{w}\Bigg)-\frac{\gamma
}{w}\left\langle \nabla u,\bar{\nabla}_{Z}Z\right\rangle =nH \label{pde}%
\end{equation}
where $w=\sqrt{\gamma+|\nabla u|^{2}}$ and $H$ is computed with respect to the
orientation of $\mbox{Gr}(u)$ given by the unit normal vector $\eta$ such that
$\left\langle \eta,Z\right\rangle \leq0$. Also $\bar{\nabla}$ denotes the
Riemannian connection of $N.$

\vspace{1ex}

Given $o\in\Omega$ let $r>0$ be such that $r<i(o)$ where $i(o)$ is the
injectivity radius of $M$ at $o$. We denote by $B_{r}(o)$ the geodesic ball
centered at $o$ and radius $r$ which closure is contained in $\Omega$. It is
proved in \cite{DLR}:

\begin{theorem}
{\hspace*{-1ex}}\textbf{. }\label{korevaar} Let $u\in C^{3}(B_{r}(o))$ be a
solution of the mean curvature equation (\ref{pde}). Then, there exists a
constant $L=L(u(o),r,\gamma,H)$ such that $|\nabla u(o)|\leq L$.
\end{theorem}

\begin{theorem}
\label{dhl}{\hspace*{-1ex}}\textbf{. } Let $\Omega\subset M$ be a bounded
$C^{2,\alpha}$ domain and $\Gamma=\partial\Omega.$ Denote by
\[
C\left(  \Gamma\right)  :=\left\{  \Psi\left(  t,p\right)  \text{
%TCIMACRO{\TEXTsymbol{\vert} }%
%BeginExpansion
$\vert$
%EndExpansion
}p\in\Gamma,\text{ }t\in\mathbb{R}\right\}  ,
\]
the Killing cylinder over $\Gamma$. Assume that the mean curvature $H_{\Gamma
}\ $of $C\left(  \Gamma\right)  $ is nonnegative with respect to the inner
orientation and that
\[
\inf_{\Omega}\operatorname*{Ric}\nolimits_{M}\geq-n\inf_{\Gamma}H_{\Gamma}%
^{2}.
\]
Let $H\geq0$ be such that $\inf_{\Gamma}H_{\Gamma}\geq H$. Then, given
$\varphi\in C^{0}(\Gamma)$ there exists a unique function $u\in C^{2,\alpha
}(\Omega)\cap C^{0}(\bar{\Omega})$ whose Killing graph has mean curvature $H$
with respect to the unit normal vector $\eta$ to the graph of $u$ satisfying
$\left\langle \eta,Z\right\rangle \leq0$ and $u|_{\Gamma}=\varphi$.
\end{theorem}

\bigskip

Usually Perron's method uses \emph{continuous} sub (super) solutions as
admissible functions (see \cite{GT}, Section 2.8). In our case it is
convenient to consider the more general case of \emph{lower semi-continuous
functions}. This approach is also used for proving the existence of
$p-$harmonic functions (\cite{HKM}, \cite{HV}).

\medskip

We denote by $C_{b}^{0}\left(  \overline{M}\right)  $ the space of lower
semi-continuous functions on $\overline{M}.$ Recall that $\sigma\in C_{b}%
^{0}\left(  \overline{M}\right)  $ is a subsolution for $Q_{H}$ in $M$ if,
given a bounded subdomain $\Lambda\subset M,$ if $u\in C^{0}\left(
\overline{\Lambda}\right)  \cap C^{2}\left(  \Lambda\right)  $ satisfies
$Q_{H}\left(  u\right)  =0$ in $\Lambda$ and $\sigma|_{\partial\Lambda}\leq
u|_{\partial\Lambda}$ then $\sigma\leq u.$ Supersolution is defined by
replacing \textquotedblleft$\leq$\textquotedblright\ by \textquotedblleft%
$\geq$\textquotedblright$.$

\medskip

The validity of Perron's method can be proved for general Killing graphs and
general domains of a Riemannian manifold. However, we state and prove it only
on the case that the domain is a whole Cartan-Hadamard manifold and the
boundary data is prescribed at infinity, case that we are interested in this paper.

\medskip

Let $N$ be a Cartan-Hadamard manifold, that is, $N$ simply connected with
sectional curvature $K_{N}\leq0.$ Then it is well defined the asymptotic
boundary $\partial_{\infty}N$ of $N$, a set of equivalence classes of geodesic
rays of $N$. We recall that two geodesic rays $\alpha,\beta:\left[
0,\infty\right)  \rightarrow N$ are in the same class if there is $C>0$ such
that $d(\alpha(t),\beta(t))\leq C$ for all $t\in\left[  0,\infty\right)  $
where $t$ is the arc length. With the so called cone topology $\overline
{N}:=N\cup\partial_{\infty}N$ is a compactification of $N$ (see \cite{BO})$.$
Since $M$ is totally geodesic in $N,$ $M$ is also a Cartan-Hadamard manifold
and we may naturally consider $\partial_{\infty}M$ as a subset of
$\partial_{\infty}N.$ We assume that the flow $\Psi$ of the Killing field $Z$
extends continuously to $\Psi\colon\mathbb{R}\times\partial_{\infty
}N\rightarrow\partial_{\infty}N$.

\begin{lemma}
\label{comp}If $\sigma\in C_{b}^{0}\left(  \overline{M}\right)  ,$
$\overline{M}=M\cup\partial_{\infty}M,$ is a subsolution and $w\in
C^{2}\left(  M\right)  \cap C^{0}\left(  \overline{M}\right)  $ a solution of
of $Q_{H}$ in $M$ such that $\sigma|_{\partial_{\infty}M}\leq w|_{\partial
_{\infty}M}$ then $\sigma\leq w$ in $\overline{M}.$
\end{lemma}

\begin{proof}
If $\sigma(p)>w(p)$ at some point $p\in M$ then there is $t>0$ small enough
such that $w_{t}:=\Psi\left(  t,w\right)  <\sigma$ in a neighborhood $U$ of
$p.$ Assume that $U$ is the biggest neighborhood of $p$ such that
$w_{t}<\sigma.$ This neighborhood $U$ has to be bounded since $w_{t}%
|_{\partial_{\infty}M}>\sigma|_{\partial_{\infty}M}$. But this is a
contradiction since $w_{t}|_{\partial U}=\sigma|_{\partial U},$ and $w_{t}$ is
a solution of $Q_{H}=0.$ This proves the lemma.
\end{proof}

\medskip

\medskip We next prove Perron's method. It is stated in a form which is
convenient for our later use.

\begin{theorem}
\label{perron}Let $N$ be a Cartan-Hadamard manifold. Let $\phi\in C^{0}\left(
\partial_{\infty}M\right)  $ and $H\in\mathbb{R}$ be given$.$ Assume that
there are a subsolution $\sigma\in C_{b}^{0}\left(  \overline{M}\right)  $ and
a solution $w\in C^{2}\left(  M\right)  \cap C^{0}\left(  \overline{M}\right)
$ of $Q_{H}$ in $M$ such that $\sigma|_{\partial_{\infty}M}\leq\phi\leq
w|_{\partial_{\infty}M}.$ Set%

\[
S_{\phi}=\left\{  v\in C_{b}^{0}\left(  \overline{M}\right)  \text{
%TCIMACRO{\TEXTsymbol{\vert} }%
%BeginExpansion
$\vert$
%EndExpansion
}v\text{ is a subsolution of }Q_{H}\text{ such that }v|_{\partial_{\infty}%
M}<\phi\right\}
\]
and%
\[
u_{\phi}(x):=\sup_{v\in S_{\phi}}v(x),\text{ }x\in\overline{M}.
\]
Then $u_{\phi}\in C^{\infty}\left(  M\right)  $ and $Q_{H}\left(  u_{\phi
}\right)  =0.$
\end{theorem}

\begin{proof}
From the hypothesis and Lemma \ref{comp} it follows that $u_{\phi}$ is a well
defined function on $\overline{M}.$ Given $x\in M,$ let $B_{r}(x)\subset M$ be
the open geodesic ball centered at $x$ and with radius $r>0$ in $M.$ From the
Hessian comparison theorem, for $r=r_{x}$ sufficiently small, if
\[
C_{r_{x}}:=\left\{  \Psi\left(  t,p\right)  \text{
%TCIMACRO{\TEXTsymbol{\vert} }%
%BeginExpansion
$\vert$
%EndExpansion
}p\in\partial B_{r_{x}}(x),\text{ }t\in\mathbb{R}\right\}  ,
\]
then $H_{x}\geq\left\vert H\right\vert ,$ where $H_{x}$ is the mean curvature
of $C_{r_{x}}$ with respect to the inner orientation$\ $and, clearly,
$\inf_{B_{r_{x}}}\operatorname*{Ric}\nolimits_{M}\geq-n\inf_{C_{r_{x}}}%
H_{x}^{2}.$

Now, take a sequence $s_{m}\in C_{b}^{0}\left(  \overline{M}\right)  $ of
subsolutions such that $s_{m}|_{\partial_{\infty}M}\leq\phi$ and $\lim
_{m}s_{m}(x)=u_{\phi}(x).$ For a fixed $m,$ let $f_{n}\in C^{0}\left(
\overline{M}\right)  $ be a decreasing sequence of functions converging
pointwise to $s_{m}$ from above. By Theorem \ref{dhl} there is a solution
$v_{n,x}\in C^{0}\left(  \overline{B_{r_{x}}(x)}\right)  \cap C^{\infty
}\left(  B_{r_{x}}(x)\right)  $ of $Q_{H}=0$ in $B_{r_{x}}(x)$ such that
$v_{n,x}|_{\partial B_{r_{x}}(x)}=f_{n}|_{\partial B_{r_{x}}(x)}.$ Using the
compactness of $\{v_{n,x}\}$ on compact subsets of $B_{r_{x}}(x)$ that follows
from Theorem \ref{korevaar} it follows that $v_{n,x}$ contains a subsequence
converging uniformly on compact subsets of $B_{r_{x}}(x)$ to a solution
$w_{m,x}\in C^{\infty}\left(  B_{r_{x}}(x)\right)  $ of $Q_{H}=0.$ We now
define the CMC $H$ lift $t_{m}\in C_{b}^{0}\left(  \overline{M}\right)  $ of
$s_{m}$ by setting%
\[
t_{m}(y)=\left\{
\begin{array}
[c]{ll}%
s_{m}(y) & \text{if }y\in M\backslash B_{r_{x}}\left(  x\right) \\
w_{m,x}(y) & \text{if }y\in B_{r_{x}}\left(  x\right)  .
\end{array}
\right.
\]
\ Using again Theorem \ref{korevaar} it follows that $w_{m,x}$ contains a
subsequence converging uniformly on compact subsets of $B_{r_{x}}(x)$ to a
solution $w_{x}\in C^{\infty}\left(  B_{r_{x}}(x)\right)  $ of $Q_{H}=0.$ As
in \cite{GT}, Section 2.8, we may prove that $w_{x}=u_{\phi}|_{B_{r_{x}}(x)}$.
This proves that $u_{\phi}\in C^{\infty}\left(  M\right)  $ and satisfies
$Q_{H}\left(  u_{\phi}\right)  =0.$
\end{proof}

\section{\label{hy}Asymptotic Plateau's problem for CMC hypersurfaces in
$\mathbb{H}^{n+1}$}

\qquad Let $\mathbb{H}^{n+1},$ $n\geq2,$ be the hyperbolic space with
sectional curvature $-1$. It is well known that there is a conformal
diffeomorphism $F$ between $\overline{\mathbb{H}}^{n+1}$ and the unit closed
ball $\overline{B}$ of $\mathbb{R}^{n+1}$, and any Killing vector field in $B$
is the restriction of a conformal vector field of $\mathbb{R}^{n+1}$ that
leaves $B$ invariant (see \cite{C}). Then, any Killing field of $\mathbb{H}%
^{n+1}$ is well defined in $\partial_{\infty}\mathbb{H}^{n+1}$ and its flow
extends continuously to $\partial_{\infty}\mathbb{H}^{n+1}$. A $k-$dimensional
sphere $E$ of $\partial_{\infty}\mathbb{H}^{n+1}$ is defined as $E=F^{-1}%
(\widetilde{E})$ where $\widetilde{E}\subset\partial B$ is an usual
$k-$dimensional sphere, that is, the intersection between $\partial B$ and a
$\left(  k+1\right)  -$dimensional linear subspace of $\mathbb{R}^{n+1}.$ If
$k=2$ then $E$ is a circle and if $k=n$ then $E$ is a hypersphere.

\medskip

\begin{theorem}
\label{hyp}Let $p_{1},p_{2}$ be two distinct points in the asymptotic boundary
$\partial_{\infty}\mathbb{H}^{n+1}$ of $\mathbb{H}^{n+1}$ and $\Gamma
\subset\partial_{\infty}\mathbb{H}^{n+1}$be\ a compact embedded topological
hypersurface not passing neither through $p_{1}$ nor $p_{2}$ and that
intersects transversely any arc of circle of $\partial_{\infty}\mathbb{H}%
^{n+1}$ having $p_{1}$ and $p_{2}$ as ending points. Let $\left\vert
H\right\vert <1$ be given. Then there exists an unique properly embedded,
complete $C^{\infty}$ hypersurface $S$ of $\mathbb{H}^{n+1}$ with CMC $H$ such
that $\partial_{\infty}S=\Gamma$ and $\overline{S}=S\cup\Gamma$ is an embedded
compact topological hypersurface of $\overline{\mathbb{H}}^{n+1}.$

Moreover there exists a totally geodesic hypersurface $\mathbb{H}^{n}$ of
$\mathbb{H}^{n+1}$, a Killing field $X$ which orbits are hypercycles
orthogonal to $\mathbb{H}^{n}$ and a function $u\in C^{\infty}\left(
\mathbb{H}^{n}\right)  \cap C^{0}\left(  \overline{\mathbb{H}}^{n}\right)  $
such that
\[
\overline{S}=\left\{  \Psi_{X}(u(x),x)\text{
%TCIMACRO{\TEXTsymbol{\vert} }%
%BeginExpansion
$\vert$
%EndExpansion
}x\in\overline{\mathbb{H}}^{n}\right\}  ,
\]
where $\Psi_{X}$ is the flow of $X.$
\end{theorem}

\medskip

\begin{theorem}
\label{par}Let $p\ $be a point in the asymptotic boundary $\partial_{\infty
}\mathbb{H}^{n+1}$ of $\mathbb{H}^{n+1}$ and $\Gamma\subset\partial_{\infty
}\mathbb{H}^{n+1}$be\ a compact embedded topological hypersurface passing
through $p.$ Assume that there are two hyperspheres $E_{1}$ and $E_{2}$ of
$\partial_{\infty}\mathbb{H}^{n+1}$ which are tangent to $p\in E_{1}\cap
E_{2}$ and such that $\Gamma$ is contained between $E_{1}$ and $E_{2},$ that
is, $\Gamma\subset U_{1}\cap U_{2}$ where $U_{i}\subset\partial_{\infty
}\mathbb{H}^{n+1}$ is the closure of the connected component of $\partial
_{\infty}\mathbb{H}^{n+1}\backslash E_{i}$ that contains $E_{j},$ $i\neq j$.

We require, moreover, that any circle of $\partial_{\infty}\mathbb{H}^{n+1}$
passing through $p$ orthogonal to $E_{1},$ intersects $\Gamma$ at one and only
one point$.$

Then, given $\left\vert H\right\vert <1,$ there exists a unique properly
embedded, complete $C^{\infty}$ hypersurface $S$ of $\mathbb{H}^{n+1}$ with
CMC $H$ such that $\partial_{\infty}S=\Gamma$ and $\overline{S}=S\cup\Gamma$
is a compact embedded topological hypersurface of $\overline{\mathbb{H}}%
^{n+1}.$

Moreover, $S$ is a parabolic graph that is, there exists a totally geodesic
hypersurface $\mathbb{H}^{n}$ of $\mathbb{H}^{n+1}$, a Killing field $Y$ which
orbits in $\mathbb{H}^{n+1}$ are horocycles orthogonal to $\mathbb{H}^{n}%
\ $and in $\partial_{\infty}\mathbb{H}^{n+1}$ are the circles orthogonal to
$E_{1}$ at $p,$ and a function $u\in C^{\infty}\left(  \mathbb{H}^{n}\right)
\cap C^{0}\left(  \overline{\mathbb{H}}^{n}\backslash\left\{  p\right\}
\right)  $ such that
\[
S=\left\{  \Psi_{Y}(u(x),x)\text{
%TCIMACRO{\TEXTsymbol{\vert} }%
%BeginExpansion
$\vert$
%EndExpansion
}x\in\mathbb{H}^{n}\right\}  ,
\]
where $\Psi_{Y}$ is the flow of $Y.$
\end{theorem}

\begin{proof}
[Proofs of Theorems \ref{hyp} and \ref{par}]In the case of Theorem \ref{hyp},
let $\gamma$ be the geodesic of $\mathbb{H}^{n+1}$ such that $\partial
_{\infty}\gamma=\left\{  p_{1},p_{2}\right\}  ,$ oriented from $p_{1}$ to
$p_{2}.$ Let $X$ be the Killing field which orbits are equidistant hypercycles
of $\gamma$ and assume that $X|_{\gamma}$ induces the same orientation of
$\gamma.$ Let $\mathbb{H}^{n}$ be a totally geodesic hypersurface orthogonal
to $\gamma$ and let $U$ be the connected component of $\mathbb{H}%
^{n+1}\backslash\mathbb{H}^{n}$ such that $p_{2}\in\partial_{\infty}U.$ We
also assume that $\Gamma\subset\partial_{\infty}U.$ Since $\Gamma$ is compact
and $p_{1},p_{2}\notin\Gamma,$ we claim that there is a hypersphere $F\subset
U$ transversal to $X$, with CMC $\left\vert H\right\vert $ when oriented with
a normal vector field $\eta$ such that $\left\langle \eta,X\right\rangle
\leq0,$\ and such that $p_{2}$ and $\Gamma$ are in distinct connected
components of $\partial_{\infty}\mathbb{H}^{n+1}\backslash\partial_{\infty}F.$
This is easily proved by considering the half-space model of $\mathbb{H}%
^{n+1},$ taking $p_{1}=0,$ $p_{2}=\left\{  x_{n+1}=\infty\right\}  $ and
$\mathbb{H}^{n}$ as the half-sphere centered at $0$ with Euclidean radius $1.$
Then, since $\Gamma\subset\partial_{\infty}\mathbb{H}$ is compact and
$p_{2}\notin\Gamma$ it follows that $\Gamma$ is compact in $\mathbb{R}%
^{n}:=\left\{  x_{n+1}=0\right\}  .$ Now, take a hypersphere $E$ of
$\mathbb{R}^{n}$ centered at the origin $0$, containing $\Gamma$ in its
interior, and choose $F$ as the hypersphere of $\mathbb{H}^{n+1}$ with CMC
$\left\vert H\right\vert $ with respect to the unit normal vector field
pointing to $0$ and having $E$ as asymptotic boundary.

The orbits of $X$ on $\partial_{\infty}\mathbb{H}^{n+1}$ are arc of circles
(up to a conformal map) from $p_{1}$ to $p_{2}.$ Therefore, from the
hypothesis it follows that $\Gamma$ is the $X-$Killing graph of a function
$\phi_{X}\in C^{0}\left(  \partial_{\infty}\mathbb{H}^{n}\right)  .$ Also, $F$
is the graph of a positive solution $w_{X}\in C^{\infty}\left(  \mathbb{H}%
^{n}\right)  \cap C^{0}\left(  \overline{\mathbb{H}^{n}}\right)  $ of
$Q_{H}=0$ in $\mathbb{H}^{n}.$

In case of Theorem \ref{par} let $\mathbb{H}^{n}$ be the a totally geodesic
hypersurface such that $\partial_{\infty}\mathbb{H}^{n}=E_{1}.$ Let $Y$ be a
Killing field orthogonal to $\mathbb{H}^{n}$ which orbits are horocycles
orthogonal do $\mathbb{H}^{n}$ and such that $Y\left(  p\right)  =0.$ Let $F$
be a hypersphere of $\mathbb{H}^{n+1}$ such that $\partial_{\infty}F=E_{2}$
and that has constant mean curvature $\left\vert H\right\vert $ with respect
to a unit normal vector field $\eta$ such that $\left\langle \eta
,Y\right\rangle \leq0.$ The orbits of $Y$ on $\partial_{\infty}\mathbb{H}%
^{n+1}\backslash\left\{  p\right\}  $ are of the form $C\backslash\left\{
p\right\}  $ where $C$ is a circle passing through $p$ and orthogonal to
$E_{1}$ and $E_{2}$\ at $p.$ Therefore, from the hypothesis it follows that
$\Gamma$ is the $Y-$Killing graph of a function $\phi_{Y}\in C^{0}\left(
\partial_{\infty}\mathbb{H}^{n}\backslash\left\{  p\right\}  \right)  .$ Also,
$F$ is the graph of a positive solution $w_{Y}\in C^{\infty}\left(
\mathbb{H}^{n}\right)  \cap C^{0}\left(  \overline{\mathbb{H}^{n}}%
\backslash\left\{  p\right\}  \right)  $ of $Q_{H}=0$ in $\mathbb{H}^{n}.$

In the sequel, let $Z$ be either $X$ or $Y$, $\phi$ either $\phi_{X}$ or
$\phi_{Y}$ and $w$ either $w_{X}$ or $w_{Y}.$ Denote by $\Psi$ the flow of $Z$
(note that since $Y(p)=0$, $\Psi(p,t)=p$ for all $t\in\mathbb{R}$ in the case
that $Z=Y$).

Obviously $\sigma=0$ is a subsolution. Then $\sigma|_{\partial_{\infty
}\mathbb{H}^{n}}\leq\phi\leq w|_{\partial_{\infty}\mathbb{H}^{n}}$ and it
follows that the function $u=u_{\phi}$ defined in Theorem \ref{perron} is in
$C^{\infty}\left(  \mathbb{H}^{n}\right)  $ and satisfies $Q_{H}\left(
u\right)  =0.$ We prove now that $u$ extends continuously to $\partial
_{\infty}\mathbb{H}^{n}$ and that $u|_{\partial_{\infty}\mathbb{H}^{n}}=\phi$
when $Z=X,$ and that $u$ extends continuously to $\partial_{\infty}%
\mathbb{H}^{n}\backslash\left\{  p\right\}  $ and that $u|_{\partial_{\infty
}\mathbb{H}^{n}\backslash\left\{  p\right\}  }=\phi$ when $Z=Y.$

Consider a point $x_{0}\in\partial_{\infty}\mathbb{H}^{n}$ and assume that
$p\neq x_{0}$ if $Z=Y.$ We first define, inductively, a sequence $\sigma
_{k}\in C_{b}^{0}\left(  \overline{\mathbb{H}}^{n}\right)  $ of subsolutions,
which graph we denote by $g_{k},$ satisfying $0\leq\sigma_{k}|_{\partial
_{\infty}\mathbb{H}^{n}}\leq\phi$ (that is, $\sigma_{k}\in S_{\phi}$ in the
notation of Theorem \ref{perron}) and such that $\phi\left(  x_{0}\right)
=\lim_{k}\sigma_{k}(x_{0}).$ Choose $\sigma_{0}=\sigma,$ the null function.
Assuming that $\sigma_{k}$ is defined, $k\geq0,$ let $E_{k}$ be the biggest
hypersphere of $\partial_{\infty}\mathbb{H}^{n+1}$ centered at $\Psi
(\sigma_{k}(x_{0}),x_{0})$ such that $E_{k}$ is contained in the closure of
the connected component of $\partial_{\infty}\mathbb{H}^{n+1}\backslash\Gamma$
that contains $\partial_{\infty}\mathbb{H}^{n}.$ Let $T_{k}$ be the totally
geodesic hypersphere of $\mathbb{H}^{n+1}$ such that $\partial_{\infty}%
T_{k}=E_{k}.$ $T_{k}$ is the union of two subdomains $T_{k}^{\pm}$ of $T_{k}$
satisfying $T_{k}^{+}\cap T_{k}^{-}=\partial T_{k}^{+}\cap\partial T_{k}^{-},$
both of them being graphs over domains $D_{k}^{\pm}$ of $\mathbb{H}^{n}$ and
with one, say $T_{k}^{+},$ such that $\partial_{\infty}T_{k}^{+}\cap\Gamma
\neq\varnothing.$ Assuming that $T_{k}^{+}$ is the graph of a function
$\alpha_{k}$ defined in $D_{k}^{+}\subset\mathbb{H}^{n},$ we define
$\sigma_{k+1}$, the semi-continuous minimal lift of $\sigma_{k}$ by
$\alpha_{k}$, as%
\begin{equation}
\sigma_{k+1}\left(  x\right)  =\left\{
\begin{array}
[c]{ll}%
\sigma_{k}(x) & \text{if }x\notin D_{k}^{+}\\
\alpha_{k}(x) & \text{if }x\in D_{k}^{+}.
\end{array}
\right.  \label{sub}%
\end{equation}

We now define in a similar way a sequence of supersolutions $w_{k}$ with graph
denoted by $G_{k}$ such that $w\geq w_{k}|_{\partial_{\infty}\mathbb{H}^{n}%
}\geq\phi$ and $\phi\left(  x_{0}\right)  =\lim_{k}w_{k}(x_{0})$ as follows$.$
We set $w_{0}=w.$ Assuming that $w_{k}$ is defined, $k\geq0,$ we define the
$H$ descent $w_{k+1}$ of $w_{k}$ as follows. Let $E_{k}$ be the biggest
hypersphere of $\partial_{\infty}\mathbb{H}^{n+1}$ centered at $\Psi
(w_{k}(x_{0}),x_{0})$ such that $E_{k}$ is contained in the closure of the
connected component of $\partial_{\infty}\mathbb{H}^{n+1}\backslash\Gamma$
that does not contain $\partial_{\infty}\mathbb{H}^{n}.$ Let $F_{k}$ be the
CMC $H$ hypersphere of $\mathbb{H}^{n+1}$ such that $\partial_{\infty}%
F_{k}=E_{k}$ and $\Gamma$ is contained in the asymptotic boundary of the
\emph{convex} connected component of $\mathbb{H}^{n+1}\backslash F_{k}$ (the
convex connected component of $\mathbb{H}^{n+1}\backslash F_{k}$ is the one
which the mean curvature vector of $F_{k}$ is pointing to. If $H=0$ then
$F_{k}$ is the only totally geodesic hypersurface such that $\partial_{\infty
}F_{k}=E_{k}$)$.$ $F_{k}$ is the union of two subdomains which are graphs over
$\widetilde{D}_{k}^{\pm}\subset\mathbb{H}^{n}$, one of them, $F_{k}^{+},$ such
that $\partial_{\infty}F_{k}^{+}\cap\Gamma\neq\varnothing.$ Assuming that
$F_{k}^{+}$ is the graph of a function $\beta_{k}$ defined in $\widetilde
{D}_{k}^{+}\subset\mathbb{H}^{n},$ we define $w_{k+1}$ as (\ref{sub}) by
replacing $\sigma_{k}$ by $w_{k},$ $D_{k}^{+}$ by $\widetilde{D}_{k}^{+}$ and
$\alpha_{k}$ by $\beta_{k}.$

We now prove, using induction on $k,$ that $u\leq w_{k}$ for all $k.$ It is
enough to prove that given $\delta\in S_{\phi}$ (using the notation of Theorem
\ref{perron}) we have $\delta<w_{k}.$ The case $k=0$ follows from Lemma
\ref{comp} since $w$ is a solution. Assume that $\delta<w_{k}$ for some
$k\geq0.$ If $\delta\geq w_{k+1}$ choose $t>0,$ $t\simeq0,$ such that
$\delta_{t}:=\Psi(t,\delta)$ satisfies $\delta_{t}<w_{k}$ and $\delta
_{t}>w_{k+1}.$ Then we have $U\subset\widetilde{D}_{k}^{+}$ where
\[
U:=\left\{  x\in\mathbb{H}^{n}\text{
%TCIMACRO{\TEXTsymbol{\vert} }%
%BeginExpansion
$\vert$
%EndExpansion
}\delta_{t}(x)>w_{k+1}(x)\right\}  .
\]
It follows that $\delta_{t}>\beta_{k}$ on $U$ and $\delta_{t}|_{\partial
U}=\beta_{k}|_{\partial U},$ a contradiction, since $\delta_{t}\in C_{b}%
^{0}\left(  \overline{\mathbb{H}}^{n}\right)  $ is a subsolution and
$\beta_{k}\in C^{\infty}\left(  U\right)  \cap C^{0}\left(  \overline
{U}\right)  $ a solution. Then $\sigma_{k}\leq u\leq w_{k}$ so that, given a
sequence $x_{m}\in\mathbb{H}^{n}$ converging to $x_{0}$ we have
\[
\sigma_{k}\left(  x_{0}\right)  =\lim_{m}\sigma_{k}\left(  x_{m}\right)
\leq\lim_{m}u\left(  x_{m}\right)  \leq\lim_{m}w_{k}\left(  x_{m}\right)
=w_{k}(x_{0}).
\]
Letting $k\rightarrow\infty$ we conclude that $\lim_{m}u\left(  x_{m}\right)
=u\left(  x_{0}\right)  =\phi\left(  x_{0}\right)  .$ This proves that $u$
extends continuously to $\partial_{\infty}\mathbb{H}^{n}$ and $u|_{\partial
_{\infty}\mathbb{H}^{n}}=\phi$ in the hyperbolic case that is, $Z=X,$ and that
$u$ extends continuously to $\partial_{\infty}\mathbb{H}^{n}\backslash\left\{
p\right\}  $ in the parabolic case ($Z=Y)$. In the parabolic case, since the
Killing graph $S$ of $u$ is contained between $F$ and $\mathbb{H}^{n},$ and
$p\in\overline{F}\cap\overline{\mathbb{H}}^{n},$ it follows that $p\in
\partial_{\infty}S.$ This proves the existence part. The uniqueness is an
immediate consequence of the maximum principle.
\end{proof}

\vspace*{-2ex}

{\renewcommand{\baselinestretch}{1} \hspace*{-20ex}\begin{tabbing}
\indent \= Jaime Ripoll \\
\> Departamento de Matematica \\
\> Univ. Federal do Rio Grande do Sul \\
\> Av. Bento Gon\c calves 9500\\
\> 91501-970 -- Porto Alegre -- RS -- Brazil\\
\> jaime.ripoll@ufrgs.br\\
\end{tabbing}}
\end{document}